\documentclass[12pt]{article}

\usepackage[latin1]{inputenc}
\usepackage{setspace}

\usepackage{algorithmic}
\usepackage{amsfonts}
\usepackage{amsmath}
\usepackage{amssymb}
\usepackage{bbm}
\usepackage{color}
\usepackage{eurosym}
\usepackage{graphicx}
\usepackage{latexsym}
\usepackage[T1]{fontenc}
\usepackage{upgreek}
\usepackage{mathtools}
\usepackage{mathdots}
\usepackage{enumerate}
\usepackage{hypernat}
\usepackage{hyperref}

\usepackage{amscd}
\usepackage{txfonts}

\newtheorem{theorem}{Theorem}

\newtheorem{lemma}{Lemma}
\newtheorem{proposition}{Proposition}

\newtheorem{example}{Example}

\newcommand{\bfb}{\mbox{$\mbox{\boldmath $b$}$}} %added by ph
\newcommand{\bfc}{\mbox{$\mbox{\boldmath $c$}$}}

\newcommand{\bff}{\mbox{$\mbox{\boldmath $f$}$}} %added by ph
\newcommand{\bfu}{\mbox{$\mbox{\boldmath $u$}$}} %added by ph
\newcommand{\bfv}{\mbox{$\mbox{\boldmath $v$}$}} %added by ph
\newcommand{\bfw}{\mbox{$\mbox{\boldmath $w$}$}} %added by ph
\newcommand{\bfx}{\mbox{$\mbox{\boldmath $x$}$}} %added by ph
\newcommand{\bfy}{\mbox{$\mbox{\boldmath $y$}$}} %added by ph
\newcommand{\bveps}{\mbox{$\mbox{\boldmath $\varepsilon$}$}} %added by ph
\newcommand{\bfzero}{\mbox{$\mbox{\boldmath $0$}$}} %added by ph

\newcommand{\bDelta}{\mbox{$\mbox{\boldmath $\Delta$}$}} %added by ph
\newcommand{\bnab}{\mbox{$\mbox{\boldmath $\nabla$}$}}

\newcommand{\bigzero}{\scalebox{0.9}{\mbox{\large $0$}}}

\begin{document}

\begin{center}

{\textbf{\Large THE NATURAL COMPONENTS OF AN  \\
AUTOREGRESSIVE TIME SERIES  \\
\rule{0cm}{0.65cm}ON BANACH SPACE   }}

%\vspace{0.5cm}
\textbf{Phil Howlett}\footnote{\textbf{Corresponding Author: Phil Howlett}, Scheduling and Control Group (SCG), Centre for Industrial and Applied Mathematics (CIAM), UniSA STEM, University of South Australia. email:~phil.howlett@unisa.edu.au,  url:~http://orcid.org/0000-0003-2382-8137.}, \textbf{Brendan K Beare}\footnote{\textbf{Brendan K Beare}, School of Economics, University of Sydney. email:~brendan.beare@sydney.edu.au, url:~https://orcid.org/0000 -0001-9146-131X.},

\end{center}

\begin{abstract}
We show that an autoregressive time series on Banach space can be canonically identified with a finite direct product of self-contained and self-determined natural components defined on separate subspaces by the spectral projections for each spectral point of the characteristic linear pencil.  We assume the resolvent of the pencil is analytic everywhere except for the spectral points which may be poles or isolated essential singularities. Each natural component is represented as a forward, backward or outward flow fixed by an upstream boundary condition in the far distant past, the far distant future, or at a nominated initial time\textemdash depending on whether the corresponding spectral point is outside, inside, or on the unit circle.
\end{abstract}

\textbf{Abstract:} We prove a generalized Granger\textendash Johansen representation theorem (GJRT) for finite or infinite order integrated auto-regressive time series on Banach space.

\rule{0cm}{0.6cm}\textbf{MSC\! [2020]:} 62M10, 91B84, 47A11, 47A56. \\

\section{Introduction}
\label{s:in}

Let $X, Y$ be complex Banach spaces and let $A_0, A_1 \in {\mathcal B}(X,Y)$ be nonzero bounded linear operators.  Consider the autoregressive equation of motion 
\begin{equation}
	\label{in:e1}
	A_0 \bfx_t + A_1 \bfx_{t-1} = \bveps_t \quad \mbox{ for all } t \in {\mathbb Z}
\end{equation}
where ${\mathbb Z}$ denotes the set of integers and the time series $\bveps = \{ \bveps_t \}_{t \in {\mathbb Z}} \in Y^{\mathbb Z}$ is known.  If $\bfx = \{\bfx_t\}_{t \in {\mathbb Z}} \in X^{\mathbb Z}$ satisfies (\ref{in:e1}) we say that $\bfx$ is an $AR(1)$ autoregressive time series generated by the known innovation series $\bveps$.  The characteristic function $A:{\mathbb C} \rightarrow {\mathcal B}(X,Y)$ is a linear pencil on the complex plane ${\mathbb C}$ defined by $A(z) = A_0 + zA_1$ for each $z \in {\mathbb C}$.  Let ${\mathbb C}_{\infty} = {\mathbb C} \cup \{ \infty\}$ denote the extended complex plane.  The spectral set for $A(z)$ is denoted by $\sigma = \sigma(A) \subset {\mathbb C}_{\infty}$ and is defined as follows.  If $\zeta \neq \infty$ and $A(\zeta)$ does not have a bounded inverse $A(\zeta)^{-1} \in {\mathcal B}(Y,X)$ then $\zeta \in \sigma$.  If $A_1$ does not have a bounded inverse $A_1^{-1} \in {\mathcal B}(Y,X)$ then $\infty \in \sigma$.  See Gohberg et al \cite[Section IV.1, p 49]{goh1} for more details.  Each $\zeta \in \sigma$ is called a spectral point for $A(z)$. 

\subsection{Assumptions}
\label{ss:a}

We assume that (i) $\sigma \neq \emptyset$ and $\sigma$ consists entirely of a finite number of isolated points; (ii) the resolvent $R(z)=A(z)^{-1} \in {\mathcal B}(Y,X)$ is analytic for all $z \in \mathbb C_\infty\setminus\sigma$; and (iii) the resolvent $R(z)$ can be represented by a Laurent series\footnote{The Laurent series representation at a spectral point $\zeta \in \sigma$ is described in Section~\ref{s:mb}.  We pay particular attention to the Laurent series representation when $\zeta = \infty \in \sigma$.} in a deleted neighbourhood of each spectral point.  We also assume
\begin{equation}
	\label{in:e2}
	\mbox{$\sum_{t \in {\mathbb Z}}$} \lVert \bveps_t \rVert r^{\lvert t \rvert} < \infty \quad \mbox{ for all } r \in (0,1).  \end{equation}
The inequality (\ref{in:e2}) implies that for each $\delta > 0$ we can find $T(\delta) \in {\mathbb N}$ such that $\lVert \bveps_t \rVert < e^{\delta \lvert t \rvert}$ for all $\lvert t \rvert > T(\delta)$.  Consequently we say that $\bveps$ is subexponential.  An equivalent condition is that $\limsup_{ \lvert t \rvert \to \infty} \| \bveps_t \|^{1/\lvert t \rvert} \leq 1$.

\subsection{Contribution}
\label{ss:cont}

We show that an autoregressive time series on Banach space defined by the equation of motion (\ref{in:e1}) and generated by a known subexponential innovation series can be canonically identified with a finite direct product of \textit{natural} components on separate subspaces.  We establish the following canonical structure.  
For each $\zeta \in \sigma$ there are corresponding spectral projections $P_\zeta \in {\mathcal B}(X)$ and $Q_\zeta \in {\mathcal B}(Y)$ such that $P_\zeta P_\eta = \bigzero_X$ and $Q_\zeta Q_\eta = \bigzero_Y$ when $\zeta, \eta \in \sigma$ and $\eta \neq \zeta$, and such that $\sum_{\zeta \in \sigma} P_\zeta = I_X$ and $\sum_{\zeta \in \sigma} Q_\zeta = I_Y$.  We define projected subspaces $X_\zeta = P_\zeta(X) \subset X$ and $Y_\zeta = Q_\zeta(Y) \subset Y$ for each $\zeta \in \sigma$.  For each $\zeta \in \sigma$ there is a component $\bfx_\zeta = P_\zeta \bfx \in X_\zeta^{\mathbb Z}$ of the time series defined by $\bfx_{\zeta,t} = P_\zeta \bfx_t$ for each $t \in {\mathbb Z}$ and a corresponding component $\bveps_\zeta = Q_\zeta \bveps \in Y_\zeta^{\mathbb Z}$ of the innovation series defined by $\bveps_{\zeta,t} = Q_\zeta \bveps_t$.   The projections induce canonical isomorphisms $X \cong \prod_{\zeta\in\sigma} X_\zeta$ and $Y \cong \prod_{\zeta\in\sigma} Y_\zeta$ via the respective linear mappings $\bfu \mapsto (\bfu_\zeta)_{\zeta\in\sigma}$ and $\bfv \mapsto (\bfv_\zeta)_{\zeta\in\sigma}$.  Applying these identifications pointwise in $t$, we associate each time series $\bfx \in X^{\mathbb Z}$ with the corresponding tuple of natural time series components $(\bfx_\zeta)_{\zeta\in\sigma} \in \prod_{\zeta\in\sigma}X_\zeta^{\mathbb Z}$ and each innovation series $\bveps \in Y^{\mathbb Z}$ with the corresponding tuple of natural innovation series components $(\bveps_\zeta)_{\zeta\in\sigma} \in \prod_{\zeta\in\sigma} Y_\zeta^{\mathbb Z}$.  The component pair $(\bfx_\zeta, \bveps_\zeta)$ satisfies the projected equation of motion\footnote{Multiply (\ref{in:e1}) on the left by $Q_\zeta$ and use the identities $Q_{\zeta}A_0 = A_0P_{\zeta}$ and $Q_{\zeta}A_1 = A_1P_{\zeta}$ for all $\zeta \in \sigma$.}
\begin{equation}
	\label{in:e3}
	A_0 \bfx_{\zeta,t} + A_1 \bfx_{\zeta, t-1} = \bveps_{\zeta, t} \quad \mbox{ for each } \zeta \in \sigma \mbox{ and all } t \in {\mathbb Z}.
\end{equation}
Thus each $\bfx_{\zeta}$ is an $AR(1)$ time series solution to (\ref{in:e3}) generated by $\bveps_\zeta$ and the pairs $(\bfx_{\zeta}, \bveps_{\zeta}) \in X_{\zeta}^{\mathbb Z} \times Y_{\zeta}^{\mathbb Z}$ are self-contained and self-determined.  The component $\bfx_\zeta$ is represented as a forward flow if $\lvert \zeta \rvert > 1$, as a backward flow if $\lvert \zeta \rvert < 1$ and as an outward flow if $\lvert \zeta \rvert = 1$.  We find an explicit solution to (\ref{in:e3}) and show that $\bfx_\zeta$ is the sum of two parts\textemdash a particular subexponential cumulation series $\bfx_{p,\zeta}$ depending only on $\bveps_\zeta$ and a complementary series $\bfx_{c,\zeta}$ determined by an arbitrary boundary constant.  When $\zeta \neq 0, \infty$ the constant lies in the projected subspace $X_\zeta$ and the flow is uniquely defined.  When $\zeta = 0, \infty$ the constant must lie in a subspace of $X_\zeta$, the algebraic core of the quasinilpotent propagating operator, in order to define a compatible forward branch for the prescribed backward flow when $\zeta = 0$, or a compatible backward branch for the prescribed forward flow when $\zeta = \infty$.  The compatible branch is uniquely defined by the boundary constant if the propagating operator is injective on the algebraic core.  The complementary series is the general solution when $\bveps_\zeta = \{\bfzero\}_{t \in {\mathbb Z}}$.

\subsection{A note about random autoregressive time series and terminology}
\label{ss:ass}

The series $\bveps = \{ \bveps_t \}_{t \in {\mathbb Z}}$ can be regarded as a \textit{realization} of a \textit{random} time series.  Let $(\Omega, \Sigma,\mu)$ be a probability space.  For each $t \in {\mathbb Z}$ let $\bveps_t: \Omega \to Y$ be a $\mu$-measurable function.  Choose $r \in (0,1)$.  If ${\mathbb E}[ \sum_{t \in {\mathbb Z}} \lVert \bveps_t \rVert r^{\lvert t \rvert} ] = \sum_{t \in {\mathbb Z}} {\mathbb E}[ \lVert \bveps_t \rVert ] r^{\lvert t \rvert} < \infty$ then $\sum_{t \in {\mathbb Z}} \lVert \bveps_t(\omega) \rVert r^{\lvert t \rvert} < \infty$ almost surely.  Apply the preceding argument to a countable sequence $\{ r_n \} \subset (0,1)$ with $r_n \uparrow 1$ and intersect the resulting probability one events. Finiteness for all $r \in (0,1)$ follows by choosing $n$ with $r<r_n$.  Now suppose $\omega \in \Omega$ is fixed \textit{a priori}.  In this case the solution to $A_0 \bfx_t(\omega) + A_1 \bfx_{t-1}(\omega) = \bveps_t(\omega)$ for all $t \in {\mathbb Z}$ is a \textit{realized} time series $\bfx = \bfx(\omega)$ generated by a \textit{realized} subexponential series $\bveps = \bveps(\omega)$.  Thus our results apply almost surely to $(\bfx(\omega), \bveps(\omega))$.

\section{Mathematical background}
\label{s:mb}

For each $\zeta \in {\mathbb C}$ let $A(z) = A_{\zeta,0} + (z-\zeta) A_1$ where $A_{\zeta,0} = A_0 + \zeta A_1$ and for each interval $J \subset [0, \infty]$ let ${\mathcal D}_{\zeta, J} = \{ z \in {\mathbb C} \mid \lvert z - \zeta \rvert \in J \}$.  Let ${\mathbb N} = \{ 1,2,3,\ldots\}$ denote the set of natural numbers and write $s + {\mathbb N} = \{ s + n \mid n \in {\mathbb N} \}$ and $s - {\mathbb N} = \{ s - n \mid n \in {\mathbb N} \}$ for each $s \in {\mathbb Z}$.  The Laurent series results are adapted from \cite{alb1,alb2}.

\paragraph{The Laurent series for $R(z)$ at $\zeta \neq \infty$.}  Let $\zeta \in \sigma$ with $\zeta \neq \infty$ and $r \in (0, \infty]$.  The resolvent $R(z)$ is analytic for $z \in {\mathcal D}_{\zeta, (0,r)}$ and can be represented in the form
\begin{equation}
	\label{mb:e1}
	R(z) = \mbox{$\sum_{k \in {\mathbb N}}$} (z - \zeta)^{-k}R_{\zeta,-k} + \mbox{$\sum_{\ell \in {\mathbb N}-1}$} (z - \zeta)^{\ell}R_{\zeta,\ell} = R_{\zeta, \mbox{\scriptsize\textup{sg}}}(z) + R_{\zeta, \mbox{\scriptsize\textup{rg}}}(z)  
\end{equation}
if and only if $\{ R_{\zeta,j}\}_{j \in {\mathbb Z}} \subset {\mathcal B}(Y,X)$ satisfy the magnitude constraints $\lim_{k \to \infty} \lVert R_{\zeta,-k} \rVert^{1/k} = 0$ and $\lim_{\ell \to \infty} \lVert R_{\zeta,\ell} \rVert^{1/\ell} \leq 1/r$ and the systems of left and right fundamental equations
\begin{equation}
	\label{mb:e2}
	R_{\zeta,j-1}A_1 + R_{\zeta,j}A_{\zeta,0} = \delta_{0,j} I_X \mbox{ and }  A_1R_{\zeta,j-1} + A_{\zeta,0}R_{\zeta,j} = \delta_{0,j} I_Y 
\end{equation}
where $\delta_{0,j} = 1$ if $j=0$ and $\delta_{0,j} = 0$ if $j \in {\mathbb Z} \setminus \{0\}$.  If these conditions are  satisfied then $(i)$ the singular part is given by $R_{\zeta, \mbox{\scriptsize\textup{sg}}}(z) = [(z - \zeta) I_X + R_{\zeta,-1}A_{\zeta,0}]^{-1}R_{\zeta,-1}$ for $z \in {\mathcal D}_{\zeta, (0, \infty)}$; $(ii)$ the regular part is given by $R_{\zeta, \mbox{\scriptsize\textup{rg}}}(z) = [I_X + (z - \zeta) R_{\zeta,0}A_1 ]^{-1}R_{\zeta,0}$ for $z \in {\mathcal D}_{\zeta, [0,r)}$; $(iii)$ $R_{\zeta,-k} = (-1)^{k-1}(R_{\zeta,-1}A_{\zeta,0})^{k-1}R_{\zeta,-1}$ for $k \in {\mathbb N}$; $(iv)$ $R_{\zeta,\ell} = (-1)^{\ell}(R_{\zeta,0}A_1)^{\ell}R_{\zeta,0}$ for $\ell \in {\mathbb N}-1$; $(v)$ the spectral projections at $\zeta$ are $P_{\zeta} = R_{\zeta,-1}A_1, P_{\zeta}^c = I_X - P_\zeta = R_{\zeta,0}A_{\zeta,0} \in {\mathcal B}(X)$ and $Q_{\zeta} = A_1 R_{\zeta,-1}, Q_{\zeta}^c = I_Y - Q_\zeta = A_{\zeta,0}R_{\zeta,0} \in {\mathcal B}(Y)$; $(vi)$ the spectral projections separate $R_{\zeta, \mbox{\scriptsize\textup{sg}}}(z)$ and $R_{\zeta, \mbox{\scriptsize\textup{rg}}}(z)$ with $R_{\zeta,-k} = P_\zeta R_{\zeta,-k}Q_\zeta$ for $k \in {\mathbb N}$ and $R_{\zeta, \ell} = P_{\zeta}^c R_{\zeta,\ell} Q_{\zeta}^c$ for $\ell \in {\mathbb N}-1$; and $(vii)$ $A_i = Q_{\zeta} A_i P_{\zeta} + Q_{\zeta}^c A_i P_{\zeta}^c$ for each $i=0,1$.

\paragraph{The Laurent series for $R(z)$ at $\zeta = \infty$.}  Let $B_0 = A_1$, $B_1 = A_0$ and $w = z^{-1}$.  The resolvent $R(z)$ has a spectral point at infinity if and only if $S(w) = (B_0 + w B_1)^{-1}$ has an isolated singularity at $w=0$ with $S(w) = \sum_{j \in {\mathbb Z}} w^j S_{0,j}$ for $w \in {\mathcal D}_{0,(0,r^{-1})}$ where $S_{0,j} \in {\mathcal B}(Y,X)$ and $r > 0$ is chosen so that $\lvert \zeta \rvert < r$ for all $\zeta \in \sigma \cap {\mathbb C}$.  The identity $R(z) = z^{-1}S(z^{-1})$ for $z \in {\mathbb C} \setminus \sigma$ shows that $R(z) = \sum_{j \in {\mathbb Z}} z^j R_{\infty,j}$ for $z \in {\mathcal D}_{0, (r,\infty)}$ with $R_{\infty,j} = S_{0,-j-1}$ for all $j \in {\mathbb Z}$.  The spectral projections at $\infty$ are $P_{\infty} = R_{\infty, 0}A_0, P_\infty^c = I_X - P_\infty = R_{\infty,-1}A_1 \in {\mathcal B}(X)$ and $Q_{\infty} = A_0R_{\infty,0}, Q_\infty^c = I_Y - Q_\infty = A_1R_{\infty,-1} \in {\mathcal B}(Y)$.

\paragraph{The global representation for $R(z)$.}  The global formul{\ae} $R(z) = \mbox{$\sum_{\zeta \in \sigma}$}\, R_{\zeta, \mbox{\scriptsize\textup{sg}}}(z)$ for $z \notin \sigma$ when $\infty \notin \sigma$; and $R(z) = \mbox{$\sum_{\zeta \in\, \sigma\, \cap\, {\mathbb C}}$}\, R_{\zeta, \mbox{\scriptsize\textup{sg}}}(z) + R_{\infty, \mbox{\scriptsize\textup{sg}}}(z)$ for $z \notin \sigma$ when $\infty \in \sigma$; were established in \cite[Section 6]{alb1}.   The singular part of $R(z)$ at $\infty$ is $R_{\infty, \mbox{\scriptsize\textup{sg}}}(z) = \sum_{\ell \in {\mathbb N}-1} z^{\ell}R_{\infty,\ell}$ for all $z \in {\mathbb C}$.

\paragraph{Representation of the spectral flow components.}  We can use either direct sums or direct products to represent the component flows. The product notation allows us to display the components explicitly.

\begin{proposition}
\label{mb:prpsn1}
For each $\zeta \in \sigma$ the spaces $X_\zeta = P_\zeta X$ and $Y_\zeta = Q_\zeta Y$ are closed Banach subspaces with $X = \bigoplus_{\zeta \in \sigma} X_\zeta$ and $Y = \bigoplus_{\zeta \in \sigma}Y_\zeta$ represented by finite internal topological direct sums.  Equivalently the finite coordinate maps $\Phi_X: X \to \prod_{\zeta \in \sigma} X_\zeta$ and $\Phi_Y: Y \to \prod_{\zeta \in \sigma} Y_\zeta$ defined by $\bfx \mapsto ( \bfx_\zeta )_{\zeta \in \sigma}$ and $\bfy \mapsto ( \bfy_\zeta)_{\zeta \in \sigma}$, where $\bfx_\zeta = P_\zeta \bfx$ and $\bfy_\zeta = Q_\zeta \bfy$, with product space norms defined by $\lVert (\bfx_\zeta)_{\zeta \in \sigma}\rVert = \max_{\zeta \in \sigma} \lVert \bfx_\zeta \rVert$ and $\lVert (\bfy_\zeta)_{\zeta \in \sigma}\rVert = \max_{\zeta \in \sigma} \lVert \bfy_\zeta \rVert$, are Banach space isomorphisms with inverse mappings defined by $(\bfx_\zeta)_{\zeta \in \sigma} \mapsto \sum_{\zeta \in \sigma} \bfx_\zeta$ and $(\bfy_\zeta)_{\zeta \in \sigma} \mapsto \sum_{\zeta \in \sigma} \bfy_\zeta$.
\end{proposition}

\textbf{Proof.}  The subspaces $X_\zeta = P_\zeta(X)$ and $Y_\zeta = Q_\zeta(Y)$ are closed because $P_\zeta$ and $Q_\zeta$ are bounded projections.  The partitions of unity $\sum_{\zeta \in \sigma} P_\zeta = I_X$ and $\sum_{\zeta \in \sigma} Q_\zeta = I_Y$ and mutual annihilation properties $P_\zeta P_\eta = \bigzero_X$ and $Q_\zeta Q_\eta = \bigzero_Y$ when $\eta \neq \zeta$ confirm the existence and uniqueness of the decompositions.  The boundedness of the coordinate map and its inverse follows from the finiteness of the spectral set and the open mapping theorem. Pointwise application of the coordinate maps for all $t \in {\mathbb Z}$ gives $X^{\mathbb Z} \cong \prod_{\zeta \in \sigma} X_\zeta^{\mathbb Z}$ and $Y^{\mathbb Z} \cong \prod_{\zeta \in \sigma} Y_\zeta^{\mathbb Z}$.  Because $\sigma$ is finite the external direct sum and Cartesian product coincide as vector spaces.  $\hfill \Box$ 

\section{Relationship to previous research and significance of the new results}
\label{s:rpr}

The representation of cointegrated $AR$ time series is of particular interest to statisticians and economists.  See \cite{bea1,how1} for recent reviews of the literature.  When $A_0, A_1 \in {\mathcal B}({\mathbb C}^m)$ are matrices and $\det A(z) \neq 0$ for $\lvert z \rvert \leq 1$ the series $\bfx$ can be modelled as a stable forward flow.  If $\det A(1) = 0$ but $\det A(z) \neq 0$ for $\lvert z \rvert \leq 1$ and $z \neq 1$ then $\bfx$ is called a \textit{unit root} process and the forward flow is no longer stable.  The Granger\textendash Johansen Representation Theorem (GJRT) for unit root processes on Euclidean space was proposed by Engle and Granger \cite{eng1} in 1987 and later proved in modified form by Johansen \cite{joh1,joh2}.  The GJRT has since been extended to Hilbert space \cite{bea2,fra1} subject to certain practical but nevertheless mathematically restrictive assumptions.  More recently the GJRT was extended to all autoregressive unit root processes on Banach space by Howlett et al.~\cite{how1}.   They showed that each unit root process $\bfx \in X^{\mathbb Z}$ can be expressed as the sum $\bfx = \bfx_{\mbox{\scriptsize sin}} + \bfx_{\mbox{\scriptsize reg}}$ of a singular pure unit root process $\bfx_{\mbox{\scriptsize sin}} \in P_1(X)^{\mathbb Z} = X_1^{\mathbb Z}$ and a regular stable time series $\bfx_{\mbox{\scriptsize reg}} \in P_1^c(X)^{\mathbb Z} = [X_1^c]^{\mathbb Z}$ on complementary subspaces defined by the spectral projections at $\zeta = 1 \in \sigma$.  This representation could also be expressed in the form $\bfx \simeq (\bfx_{\mbox{\scriptsize sin}}, \bfx_{\mbox{\scriptsize reg}}) \in X_1^{\mathbb Z} \times [X_1^c]^{\mathbb Z}$  In this paper we consider a wider class of time series and extend the two-component representation to a more expansive spectral-coordinate representation $\bfx \simeq (\bfx_\zeta)_{\zeta \in \sigma} \in \prod_{\zeta \in \sigma} X_\zeta^{\mathbb Z}$ over all spectral points.  A similar decomposition \cite{bea1} using direct sums was recently established for autoregressive time series on ${\mathbb C}^m$.  The direct product notation is preferred here because it emphasizes the separation of components and underlines the self-determined and self-contained evolution of the different components. The extension to time series on Banach space is significant because the global spectral decomposition \cite{alb1} is more difficult to justify; the spectral points may be isolated essential singularities; the spectral subspaces may be infinite-dimensional; the particular cumulation series may be infinite sums depending on all upstream innovations; the operators $(R_{0,-1}A_0), (R_{\infty,0}A_1) \in {\mathcal B}(X)$ may be quasinilpotent; and the solutions $\bfx_0, \bfx_\infty$ may include a nonzero complementary series.  There are no such solutions on ${\mathbb C}^m$.  

\section{The main results}
\label{s:mr}

For each $\bveps \in Y^{\mathbb Z}$ and each $\zeta \neq 0, \infty$ define the respective forward and backward weighted difference operators $\bDelta_{\zeta} \bveps_{\zeta,t} = \zeta \bveps_{\zeta,t+1} - \bveps_{\zeta,t}$ and $\bnab_{\zeta} \bveps_{\zeta,t} = \zeta \bveps_{\zeta,t} - \bveps_{\zeta,t-1}$ for all $t \in {\mathbb Z}$.  Negative powers of the weighted differences can be expressed as absolutely convergent sums 
\begin{equation}
\label{mr:e0b}
{\bnab_{\zeta}}^{-k} \bveps_{\zeta,t} = \mbox{$\sum_{\ell \in {\mathbb N}-1}$} \mbox{$\binom{k-1+\ell}{\ell}$} \zeta^{-k-\ell} \bveps_{\zeta, t-\ell}
\end{equation}
for $\lvert \zeta \rvert > 1$ and
\begin{equation}
\label{mr:e0f}
{\bDelta_{\zeta}}^{-k} \bveps_{\zeta, t+k} = (-1)^k \mbox{$\sum_{r \in {\mathbb N}-1}$} \mbox{$\binom{k-1+r}{r}$} \zeta^r \bveps_{\zeta, t+k+r}
\end{equation}
for $0 < \lvert \zeta \rvert < 1$.  Let $\bfx_{t_0}$ be the observed initial value.  Define the truncated innovation series $\bveps_{\zeta,t_0^+}$ and $\bveps_{\zeta,t_0^-}$ by setting 
\begin{equation*}
	\bveps_{\zeta,t_0^+,t} = \left\{\begin{array}{ll}
		\bfzero & \mbox{for } t \in t_0 + 1 - {\mathbb N} \\
		\bveps_{\zeta,t} & \mbox{for } t \in t_0 + {\mathbb N} \end{array} \right. \quad \mbox{and} \quad \bveps_{\zeta,t_0^-,t} = \left\{\begin{array}{ll}
		\bveps_{\zeta,t} & \mbox{for } t \in t_0 +1 - {\mathbb N} \\
		\bfzero & \mbox{for } t \in t_0 + {\mathbb N}. \end{array} \right.
\end{equation*}
If $\bfc \in Y$ we write $F^{-1} \bfc = \{ \bfb \in X \mid F \bfb = \bfc \}$ whether or not $F \in {\mathcal B}(X,Y)$ is invertible.  Define the noninvertible operators $G_0 = R_{0,-1}A_0 \in {\mathcal B}(X)$ when $0 \in \sigma$ and $G_\infty = R_{\infty,0}A_1 \in {\mathcal B}(X)$ when $\infty \in \sigma$ and the invertible operator $H_\zeta = \zeta I_X - R_{\zeta,-1}A_{\zeta,0} \in {\mathcal B}(X)$ when $\zeta \in \sigma$ and $\zeta \neq 0, \infty$.  In order to derive the time series solutions as directed flows it is necessary to reconcile the observed values with a compatible but unobserved upstream orbit.  If the propagating operator $G$ is not invertible the algebraic core of $G$ is a key instrument in recovering a compatible upstream orbit.

\begin{lemma}[algebraic core]
\label{mr:lem1}
Let $U$ be a Banach space and let $G \in {\mathcal B}(U)$.  Define
\begin{equation*}
	{\mathcal C}_G(U) = \{ \bfu_0 \in U \mid \exists\ \{ \bfu_p \}_{p \in {\mathbb N}} \subset U \mbox{ with } \bfu_{p-1} = G(\bfu_p) \mbox{ for all } p \in {\mathbb N} \}.
\end{equation*}
Then:
\begin{enumerate}[(i)]
\item the algebraic core ${\mathcal C}_G(U)$ is a linear subspace;
\item $G({\mathcal C}_G(U)) = {\mathcal C}_G(U)$;
\item ${\mathcal C}_G(U)$ is the largest linear subspace $M \subset U$ satisfying $G(M) = M$;
\item if $G$ is injective on ${\mathcal C}_G(U)$ then $G\vert_{{\mathcal C}_G(U)}:{\mathcal C}_G(U) \to {\mathcal C}_G(U)$ is bijective and the compatible orbit $G^{-p}\bfu$ is algebraically uniquely defined\,\footnote{The space ${\mathcal C}_G(U)$ need not be closed and there is no assertion that $G^{-1}$ is bounded on ${\mathcal C}_G(U)$.} for each $\bfu \in {\mathcal C}_G(U)$;
\item ${\mathcal C}_G(U) \subset G^\infty(U) \coloneqq \bigcap_{n \in {\mathbb N}} G^n(U)$ with ${\mathcal C}_G(U) = G^\infty(U)$ whenever $G(G^\infty(U)) = G^\infty(U)$ which is true, in particular, if $G$ is injective on $U$;
\item if $G$ is quasinilpotent with $\lVert G^p \rVert^{1/p} \to 0$ as $p \to \infty$ and $\bfu_0 \neq \bfzero$ with $\bfu_{p-1} = G \bfu_p$ for all $p \in {\mathbb N}$ then $\lVert \bfu_p \rVert^{1/p} \to \infty$ as $p \to \infty$ and the compatible orbit $\{ \bfu_p\}_{p \in {\mathbb N}}$ is not subexponential;
\item if $G$ is nilpotent ${\mathcal C}_G(U) = \{ \bfzero \}$.
\end{enumerate}
\end{lemma}

\textbf{Proof.}  Let \(M={\mathcal C}_G(U)\).

(i) The zero sequence $\{ \bfu_p\}_{p \in {\mathbb N}-1}$ with $\bfu_p = \bfzero$ for all $p \in \mathbb{N}-1$ satisfies $\bfu_{p-1} = G\bfu_p$, so $\bfzero \in M$.  If \(\bfu_0, \bfv_0 \in M\), they have compatible orbits \(\{\bfu_p\}_{p \in {\mathbb N}}\) and \(\{\bfv_p\}_{p \in {\mathbb N}}\) satisfying $\bfu_{p-1} = G\bfu_p$ and $\bfv_{p-1} = G\bfv_p$.  Since \(G(\alpha \bfu_{p}+ \beta \bfv_{p})= \alpha G\bfu_{p}+ \beta G\bfv_{p}= \alpha \bfu_{p-1}+\beta \bfv_{p-1}\) it follows that $\alpha \bfu_0 + \beta \bfv_0$ also has a compatible orbit $\{\alpha \bfu_p + \beta \bfv_p\}_{p \in {\mathbb N}}$ and so $\alpha \bfu_0 + \beta \bfv_0 \in M$ for all $\alpha, \beta \in {\mathbb C}$.  Therefore $M$ is a linear subspace.

(ii) Let \(\bfu_0 \in M\). By the definition of $M$, there exists a compatible orbit \(\{\bfu_{p}\}_{p \in {\mathbb N}}\) for $\bfu_0$ such that $\bfu_{p-1}=G \bfu_p$.  The sequence \(\{\bfu_{p+1}\}_{p \in {\mathbb N}}\) forms a compatible orbit for \(\bfu_{1}\) because \(\bfu_p = G\bfu_{p+1}\) for all \(p \in {\mathbb N}\).  Therefore $\bfu_1 \in M \Rightarrow \bfu_0 = G \bfu_1 \in G(M)\).  Thus $M \subset G(M)$.  Let \(\bfv = G \bfu_0 \in G(M)\) for some \(\bfu_0 \in M\).  Since \(\bfu_0 \in M\), it has a compatible orbit \(\{\bfu_p \}_{p \in {\mathbb N}}\).  Therefore \(\bfv_0 = \bfv\) and \(\bfv_p = \bfu_{p-1}\) for \(p \in {\mathbb N}\) defines a compatible orbit $\{ \bfv_p \}_{p \in {\mathbb N}}$ for $\bfv$.  Hence \(\bfv \in M\).  Thus $G(M) \subset M$.  It follows that $G(M) = M$.

(iii) Let \(L \subset U\) be any linear subspace such that \(G(L) = L\). Let $\bfu_0 \in L$.  Because \(G(L) = L\) there exists \(\bfu_1 \in L\) such that \(G \bfu_1 = \bfu_0\).  By mathematical induction there exists \(\{\bfu_p\}_{p \in \mathbb{N}} \subset L\) such that \(\bfu_{p-1} = G \bfu_p\) for all \(p \in \mathbb{N}\).  Thus, \(\bfu_0 \in M\). It follows that \(L \subset M\).  Therefore \(M\) is the largest linear subspace with $G(M)=M$.

(iv) If $\bfu \in M$ there is a unique image $\bfw = G \bfu \in G(M)$ for all $\bfu \in M$.  Now suppose $\bfw \in G(M)$.  If $G$ is injective on $M$ then $G \bfv = G \bfu \Rightarrow \bfv = \bfu$.  Let $\bfu = G^{-1} \bfw$ denote the unique preimage for each $w \in G(M)$.  Thus $G\vert_M: M \to M$ is bijective.  Since $M = G(M)$ we can apply these arguments repeatedly to show that each $\bfu \in M$ has a collection of uniquely defined preimages $G^{-p} \bfu$ for all $p \in {\mathbb N}$.   

(v) If $G$ is injective on $U$ then $\bfv \in G^\infty(U) \subset U \Rightarrow \bfu = G^{-1}\bfv$ is uniquely defined and $\bfu \in G^\infty(U) \Rightarrow \bfv = G \bfu \in G(G^\infty(U)) \Rightarrow G^\infty(U) \subset G(G^\infty(U))$.  Now the obvious inclusion $G(G^\infty(U)) \subset G^\infty(U) \Rightarrow G^\infty(U) = G(G^\infty(U))$.  From (iii) the condition $G(G^\infty(U)) = G^\infty(U) \Rightarrow G^\infty(U) \subset M$.  If $\bfu_0 \in M$ there exists $\{ \bfu_p \}_{p \in {\mathbb N}} \subset U$ with $\bfu_0 = G^p \bfu_p$ for all $p \in {\mathbb N}$.  Therefore $M \subset  G^\infty(U)$.  Thus $G^\infty(U) = M$.

(vi) We have $G^p \bfu_p = \bfu_0 \neq \bfzero$ for all $p \in {\mathbb N}$.  Therefore $\lVert \bfu_0 \rVert \leq \lVert G^p \rVert \cdot \lVert \bfu_p \rVert$.  If $\lVert G^p \rVert^{1/p} \to 0$ as $p \to \infty$ then $\lVert \bfu_p \rVert^{1/p} \geq \| \bfu_0 \|^{1/p} /\lVert G^p \rVert^{1/p} \to \infty$ as $p \to \infty$.

(vii) If $G$ is nilpotent there exists $q \in {\mathbb N}$ such that $G^p \bfu = \bfzero$ for all $\bfu \in U$ and all $p \in {\mathbb N}+q$.  Therefore ${\mathcal C}_G(U) = \{ \bfzero \}$.   $\hfill \Box$

\begin{theorem}
	\label{mr:thm1}
	Choose $t_0 \in {\mathbb Z}$.
	\begin{enumerate}[(i)]
		\item If $\zeta \in \sigma$ and $\lvert \zeta \rvert \in (1, \infty)$ each $\bfc_{\zeta, t_0, -\infty} \in X_\zeta$ defines a unique forward flow solution
	\begin{equation}
		\label{mr:e1}
		\bfx_{\zeta,t} = \mbox{$\sum_{k \in {\mathbb N}}$} (-1)^k R_{\zeta,-k}\, {\bnab_{\zeta}}^{-k} \bveps_{\zeta,t} + {H_\zeta}^{-(t-t_0)} \bfc_{\zeta,t_0,-\infty} \in X_\zeta \quad \mbox{ for all } t \in {\mathbb Z}
		\end{equation}
	where $\bfc_{\zeta,t_0,-\infty} = \lim_{r \rightarrow \infty} {H_\zeta}^{-r-t_0} \bfx_{\zeta,-r}$ represents a boundary value in the far distant past.  The flow is subexponential if and only if $\bfc_{\zeta, t_0, -\infty} = \bfzero$.
	
	\item If $\zeta \in \sigma$ and $\lvert \zeta \rvert \in (0,1)$ each $\bfc_{\zeta, t_0, \infty} \in X_\zeta$ defines a unique backward flow solution
	\begin{equation}
		\label{mr:e2}
		\bfx_{\zeta,t} = \mbox{$\sum_{k \in {\mathbb N}}$} (-1)^k R_{\zeta,-k}\, {\bDelta_{\zeta}}^{-k} \bveps_{\zeta,t+k} + {H_\zeta}^{-(t-t_0)} \bfc_{\zeta,t_0,\infty} \in X_\zeta  \quad \mbox{ for all } t \in {\mathbb Z}
	\end{equation}
	where $\bfc_{\zeta, t_0, \infty} = \lim_{r \rightarrow \infty} {H_\zeta}^{r - t_0} \bfx_{\zeta, r}$ represents a boundary value in the far distant future.  The flow is subexponential if and only if $\bfc_{\zeta, t_0, \infty} = \bfzero$.
	
	\item If $\zeta \in \sigma$ and $\lvert \zeta \rvert = 1$ each $\bfc_{\zeta,t_0} = \bfx_{\zeta,t_0} \in X_\zeta$ defines a unique outward flow solution
	\begin{equation}
		\label{mr:e3}
		\bfx_{\zeta,t} = \left\{ \begin{array}{ll}
			\mbox{$\sum_{k \in {\mathbb N}}$} (-1)^k R_{\zeta,-k}\, {\bnab_{\zeta}}^{-k} \bveps_{\zeta,t_0^+,t} + {H_\zeta}^{-(t-t_0)} \bfc_{\zeta,t_0} \in X_\zeta & \mbox{ for } t \in t_0 + {\mathbb N} \\
			\mbox{$\sum_{k \in {\mathbb N}}$} (-1)^k R_{\zeta,-k}\, {\bDelta_{\zeta}}^{-k} \bveps_{\zeta,t_0^-,t+k} + {H_\zeta}^{(t_0-t)} \bfc_{\zeta,t_0} \in X_\zeta & \mbox{ for } t \in t_0 - {\mathbb N} \end{array} \right.
	\end{equation}
	where $t_0 + {\mathbb N} = \{ t \in {\mathbb Z} \mid t > t_0 \}$ and $t_0 - {\mathbb N} = \{ t \in {\mathbb Z} \mid t < t_0\}$ and where $\bfx_{\zeta,t_0}$ is the initial value.  The flow is always subexponential with ${\bnab_{\zeta}}^{-k} \bveps_{\zeta,t_0^+,t} = \bfzero$ if $t \in t_0 - {\mathbb N} = \{ t \in {\mathbb Z} \mid t < t_0 \}$ and ${\bDelta_{\zeta}}^{-k} \bveps_{\zeta,t_0^-,t+k} = \bfzero$ if $t + k \in t_0+{\mathbb N} = \{ t + k \in {\mathbb Z} \mid t+k > t_0\}$.
	
	\item If $\infty \in \sigma$ the projected equation takes the form $\bfx_{\infty,t} = R_{\infty,0} \bveps_{\infty,t} - G_\infty \bfx_{\infty,t-1}$ where $G_\infty = (R_{\infty,0}A_1)$.  This recursion propagates naturally in the forward direction.  Recovery of a compatible flow for $t \in t_0 - {\mathbb N} = \{ t \in {\mathbb Z} \mid t < t_0\}$ upstream from the observed value $\bfx_{\infty,t_0}$ requires generation of successive preimages under $G_\infty$.  In this context each $\bfc_{\infty, t_0, -\infty} = \lim_{r \rightarrow \infty} (-1)^{r+t_0} G_\infty^{r+t_0} \bfx_{\infty,-r} \in {\mathcal C}_{G_\infty}(X_\infty)$ represents a compatible boundary value in the far distant past and defines a unique forward flow for $t \in t_0 -1 + {\mathbb N} = \{ t \in {\mathbb Z} \mid t \geq t_0 \}$ with
	\begin{equation}
		\label{mr:e4}
		\bfx_{\infty,t} = \mbox{$\sum_{\ell \in {\mathbb N}-1}$} R_{\infty,\ell}\, \bveps_{\infty,t-\ell} + (-1)^{t-t_0} G_\infty^{\,t-t_0} \bfc_{\infty, t_0, -\infty} \in X_\infty.
	\end{equation}
	If the propagating operator $G_\infty\vert_{{\mathcal C}_{G_\infty}(X_\infty)}$ is injective then $\bfx_{\infty,t}$ is uniquely determined by $(\ref{mr:e4})$ for all $t \in {\mathbb Z}$ and the flow is subexponential if and only if $\bfc_{\infty, t_0, -\infty} = \bfzero$.  If $G_\infty\vert_{{\mathcal C}_{G_\infty}(X_\infty)}$ is not injective $(\ref{mr:e4})$ is no longer a valid solution per se when $t \in t_0 - {\mathbb N} = \{ t \in {\mathbb Z} \mid t < t_0\}$ but rather provides a framework within which a valid solution can be constructed.  The flow is no longer uniquely defined and one must construct a compatible upstream flow $\bfx_{c,\infty,t} = (-1)^{t-t_0} \bfu_{t_0-t}$ for $t \in t_0 - {\mathbb N}$ by setting $\bfu_0 = \bfc_{\infty,t_0,-\infty}$ and choosing $\{ \bfu_p \}_{p \in {\mathbb N}}$ with $\bfu_{p-1} = G_\infty \bfu_p$ for all $p \in {\mathbb N}$.  The propagating operator $G_\infty$ is quasinilpotent and hence is not invertible.
	
	\item If $0 \in \sigma$ the projected equation takes the form $\bfx_{0,t} = R_{0,-1}\bveps_{0,t+1} - G_0 \bfx_{0,t+1}$ where $G_0 = (R_{0,-1}A_0)$.  This recursion propagates naturally in the backward direction.  Recovery of a compatible flow for $t \in t_0 + {\mathbb N} = \{ t \in {\mathbb Z} \mid t > t_0\}$ upstream from the observed value $\bfx_{0,t_0}$ requires generation of successive preimages under $G_0$.  In this context each $\bfc_{0,t_0,\infty} = \lim_{r \rightarrow \infty} (-1)^{r-t_0} G_0^{r-t_0} \bfx_{0,r} \in {\mathcal C}_{G_0}(X_0)$ represents a compatible boundary value in the far distant future and defines a unique backward flow for $t \in t_0 +1 - {\mathbb N} = \{ t \in {\mathbb Z} \mid t \leq t_0\}$ with
	\begin{equation}
		\label{mr:e5}
		\bfx_{0,t} = \mbox{$\sum_{k \in {\mathbb N}}$} R_{0,-k}\, \bveps_{0,t+k} + (-1)^{\,t_0-t} G_0^{t_0-t} \bfc_{0, t_0, \infty} \in X_0.
	\end{equation}
	If the propagating operator $G_0\vert_{{\mathcal C}_{G_0}(X_0)}$ is injective then $\bfx_{0,t}$ is uniquely determined by $(\ref{mr:e5})$ for all $t \in {\mathbb Z}$ and the flow is subexponential if and only if $\bfc_{0, t_0, \infty} = \bfzero$.  If $G_0\vert_{{\mathcal C}_{G_0}(X_0)}$ is not injective $(\ref{mr:e5})$ is no longer a valid solution per se when $t \in t_0 + {\mathbb N} = \{ t \in {\mathbb Z} \mid t > t_0\}$ but rather provides a framework within which a valid solution can be constructed.  The flow is no longer uniquely defined and one must construct a compatible upstream flow $\bfx_{c,0,t} = (-1)^{t-t_0} \bfu_{t-t_0}$ for $t \in t_0 + {\mathbb N}$ by setting $\bfu_0 = \bfc_{0,t_0,\infty}$  and choosing $\{ \bfu_p \}_{p \in {\mathbb N}}$ with $\bfu_{p-1} = G_0 \bfu_p$ for all $p \in {\mathbb N}$.  The propagating operator $G_0$ is quasinilpotent and hence is not invertible.
	\end{enumerate} 
\end{theorem} 

\textbf{Proof.}  Some standard algebra has been relegated to the Appendix.

\paragraph{Part $(i)$.} If $\zeta \in \sigma$ and $\lvert \zeta \rvert \in (1, \infty)$ then Lemma~\ref{sar:lem1} (i) shows that $\bfx_\zeta$ satisfies the forward recursion $\bfx_{\zeta,t} = -H_\zeta^{-1}R_{\zeta,-1}\bveps_{\zeta,t} + H_\zeta^{-1}\bfx_{\zeta,t-1}$ for all $t \in {\mathbb Z}$.  Suppose $t \in {\mathbb Z}$ is fixed.  Repeated application of the recursion shows that
\begin{equation}
	\label{mr:e6}
	\bfx_{\zeta,t} = (-1) \mbox{$\sum_{\ell=0}^{t+r-1}$} {H_\zeta}^{-\ell-1} R_{\zeta,-1} \bveps_{\zeta,t-\ell} + {H_\zeta}^{-r-t} \bfx_{\zeta,-r} \quad \mbox{ for all } r \in -t + {\mathbb N}.
\end{equation}
The expression $[z I_X - H_\zeta]^{-1} = [(z - \zeta)I_X + R_{\zeta,-1}A_{\zeta,0}]^{-1}$ is analytic for all $z \neq \zeta$ because $R_{\zeta,-1}A_{\zeta,0}$ is quasinilpotent so the Maclaurin series $[ z I_X - H_\zeta]^{-1} = (-1) \mbox{$\sum_{\ell \in {\mathbb N}-1}$} z^{\ell} {H_\zeta}^{-\ell-1}$ converges for $\lvert z \rvert < \lvert \zeta \rvert$ with $\lim_{\ell \to \infty} \lVert {H_\zeta}^{-\ell} \rVert^{1/\ell} = 1/\lvert \zeta \rvert < 1$.  If $\bveps$ is subexponential $\sum_{\ell \in {\mathbb N}-1} {H_\zeta}^{-\ell-1} R_{\zeta,-1} \bveps_{\zeta, t-\ell}$ converges.  Letting $r \to \infty$ in (\ref{mr:e6}) gives
\begin{equation*}
	\bfx_{\zeta,t} = (-1) \mbox{$\sum_{\ell \in {\mathbb N}-1}$} {H_\zeta}^{-\ell-1} R_{\zeta,-1} \bveps_{\zeta, t-\ell}  + {H_\zeta}^{-(t-t_0)} \bfc_{\zeta, t_0,- \infty} \quad \mbox{ for all } t \in {\mathbb Z}
\end{equation*}
where $\bfc_{\zeta,t_0, -\infty} = \lim_{r \rightarrow \infty} {H_\zeta}^{-r-t_0} \bfx_{\zeta,-r} \in X_\zeta$.  The general solution $\bfx_{\zeta,t}$ is the sum of a particular cumulation series $\bfx_{p,\zeta,t} = (-1) \mbox{$\sum_{\ell \in {\mathbb N}-1}$} {H_\zeta}^{-\ell-1} R_{\zeta,-1} \bveps_{\zeta, t-\ell}$ and a complementary series $\bfx_{c,\zeta,t} = {H_\zeta}^{-(t-t_0)} \bfc_{\zeta, t_0,- \infty}$ for all $t \in {\mathbb Z}$.  The series $\bfx_{p,\zeta}$ is subexponential because $\lim_{\ell \to \infty} \lVert {H_\zeta}^{-\ell} \rVert^{1/\ell} < 1$.  Thus $\bfx_{\zeta}$ is subexponential if and only if $\bfc_{\zeta, t_0,- \infty} = \bfzero$.  Lemma~\ref{sar:lem2}~(i) shows that $\bfx_{p,\zeta,t}$ can be rewritten as a sum of inverse powers of weighted backward differences in the form
\begin{equation*}
\bfx_{p,\zeta,t} =  \mbox{$\sum_{k \in {\mathbb N}}$} (-1)^k R_{\zeta,-k}\, {\bnab_{\zeta}}^{-k} \bveps_{\zeta,t} \quad \mbox{ for all } t \in {\mathbb Z}.
\end{equation*}
 
\paragraph{Part $(ii)$.}  If $\zeta \in \sigma$ and $\lvert \zeta \rvert \in (0,1)$ then Lemma~\ref{sar:lem1} (i) shows that $\bfx_\zeta$ satisfies the backward recursion $\bfx_{\zeta,t} = R_{\zeta,-1} \bveps_{\zeta, t+1} + H_\zeta \bfx_{\zeta, t+1}$ for all $t \in {\mathbb Z}$.  Suppose $t \in {\mathbb Z}$ is fixed.  Repeated application of the recursion gives
\begin{equation}
\label{mr:e7}
\bfx_{\zeta, t} = \mbox{$\sum_{\ell=1}^{r-t}$} {H_\zeta}^{\ell -1} R_{\zeta,-1} \bveps_{\zeta, t+\ell} + {H_\zeta}^{r-t} \bfx_{\zeta, r} \quad \mbox{ for all } r \in t+ {\mathbb N}.
\end{equation}
The expression $[z I_X - H_\zeta]^{-1} = [(z - \zeta)I_X + R_{\zeta,-1}A_{\zeta,0}]^{-1}$ is analytic for all $z \neq \zeta$ because $R_{\zeta,-1}A_{\zeta,0}$ is quasinilpotent so the Laurent series $[z I_X - H_\zeta]^{-1} = \mbox{$\sum_{\ell \in {\mathbb N}}$} z^{-\ell} {H_{\zeta}}^{\ell-1}$ converges for $\lvert z \rvert > \lvert \zeta \rvert$.  Therefore $\lim_{\ell \to \infty} \lVert {H_\zeta}^{\ell} \rVert^{1/\ell} = \lvert \zeta \rvert < 1$.  If $\bveps$ is subexponential it follows that $\sum_{\ell \in {\mathbb N}} {H_\zeta}^{\ell -1} R_{\zeta,-1} \bveps_{\zeta, t+\ell}$ converges.  Letting $r \to \infty$ in (\ref{mr:e7}) gives 
\begin{equation*}
\bfx_{\zeta, t} = \mbox{$\sum_{\ell \in {\mathbb N}}$} {H_\zeta}^{\ell -1} R_{\zeta,-1} \bveps_{\zeta, t+\ell} + {H_\zeta}^{-(t-t_0)} \bfc_{\zeta, t_0, \infty} \quad \mbox{ for all } t \in {\mathbb Z}
\end{equation*}
where $\bfc_{\zeta, t_0, \infty} = \lim_{r \rightarrow \infty} {H_\zeta}^{r - t_0} \bfx_{\zeta, r} \in X_\zeta$.  The general solution $\bfx_{\zeta,t}$ is the sum of a particular cumulation series $\bfx_{p,\zeta,t} = \mbox{$\sum_{\ell \in {\mathbb N}}$} {H_\zeta}^{\ell -1} R_{\zeta,-1} \bveps_{\zeta, t+\ell}$ and a complementary series $\bfx_{c,\zeta,t} =  {H_\zeta}^{-(t-t_0)} \bfc_{\zeta, t_0, \infty}$ for all $t \in {\mathbb Z}$.  The series $\bfx_{p,\zeta}$ is subexponential because $\lim_{\ell \to \infty} \lVert {H_\zeta}^\ell \rVert^{1/\ell} < 1$.  Thus $\bfx_{\zeta}$ is subexponential if and only if $\bfc_{\zeta, t_0, \infty} = \bfzero$. Lemma~\ref{sar:lem2}~(ii) shows that $\bfx_{p,\zeta,t}$ can be rewritten as a sum of inverse powers of weighted forward differences in the form
\begin{equation*}
\bfx_{p,\zeta,t} = \mbox{$\sum_{k \in {\mathbb N}}$}\, (-1)^k R_{\zeta,-k}\, {\bDelta_{\zeta}}^{-k} \bveps_{\zeta, t+k}  \quad \mbox{ for all } t \in {\mathbb Z}.
\end{equation*}

\paragraph{Part $(iii)$.}  If $\zeta \in \sigma$ with $\lvert \zeta \rvert = 1$ then Lemma~\ref{sar:lem1}~(i) shows that $\bfx_\zeta$ satisfies the forward recursion $\bfx_{\zeta,t} = -H_\zeta^{-1}R_{\zeta,-1}\bveps_{\zeta,t} + H_\zeta^{-1}\bfx_{\zeta,t-1}$.  Repeated application of the recursion starting at $t_0$ shows that the forward branch of the outward flow is 
\begin{equation*}
\bfx_{\zeta, t} = (-1) \mbox{$\sum_{\ell=0}^{t-t_0-1}$} H_\zeta^{-\ell-1} R_{\zeta,-1} \bveps_{\zeta, t-\ell} + H_\zeta^{-(t-t_0)} \bfc_{\zeta,t_0} \quad \mbox{ for all } t \in t_0 + {\mathbb N}
\end{equation*}
where $\bfc_{\zeta, t_0} = \bfx_{\zeta,t_0} \in X_{\zeta}$.   The branch is the sum of a particular cumulation series $\bfx_{p,\zeta,t} = (-1) \mbox{$\sum_{\ell=0}^{t-t_0-1}$} H_\zeta^{-\ell-1} R_{\zeta,-1} \bveps_{\zeta, t-\ell}$ and a complementary series $\bfx_{c,\zeta,t} = H_\zeta^{-(t-t_0)} \bfc_{\zeta,t_0}$ for all $t \in t_0 + {\mathbb N}$.  The Maclaurin series $[z I_X - H_\zeta]^{-1} = (-1) \mbox{$\sum_{\ell \in {\mathbb N}-1}$} z^\ell {H_\zeta}^{-\ell-1}$ is valid for $\lvert z \rvert < \lvert \zeta \rvert = 1$.   Therefore $\lim_{\ell \to \infty} \lVert H_\zeta^{-\ell} \rVert^{1/\ell} = 1$ and so $\bfx_{c,\zeta,t}$ is subexponential for $t \in t_0 + {\mathbb N}$.  If $\bveps$ is subexponential then $\bfx_{p,\zeta,t}$ and $\bfx_{\zeta,t}$ are also subexponential for $t \in t_0 + {\mathbb N}$.  Lemma~\ref{sar:lem2}~(iii) shows that $\bfx_{p,\zeta,t}$ can be rewritten as a sum of inverse powers of weighted truncated backward differences in the form
\begin{equation*}
\bfx_{p,\zeta,t} = \mbox{$\sum_{k \in {\mathbb N}}$} (-1)^k R_{\zeta,-k}\, {\bnab_{\zeta}}^{-k} \bveps_{\zeta,t_0^+,t} \quad \mbox{ for all } t \in t_0 + {\mathbb N}.
\end{equation*}
Lemma~\ref{sar:lem1}~(i) also shows that $\bfx_\zeta$ satisfies the backward recursion $\bfx_{\zeta,t} = R_{\zeta,-1}\bveps_{\zeta,t+1} + H_\zeta \bfx_{\zeta,t+1}$.  Repeated application of the recursion starting at $t_0$ shows that the backward branch of the outward flow is
\begin{equation*}
\bfx_{\zeta, t} = \mbox{$\sum_{\ell=1}^{-(t-t_0)}$} H_\zeta^{\ell-1} R_{\zeta,-1} \bveps_{\zeta, t+\ell} + H_\zeta^{\,t_0-t} \bfc_{\zeta,t_0} \quad \mbox{ for all } t \in t_0 - {\mathbb N}
\end{equation*}
where $\bfc_{\zeta, t_0} = \bfx_{\zeta,t_0} \in X_{\zeta}$.  The branch is the sum of a particular cumulation series $\bfx_{p,\zeta,t} = \mbox{$\sum_{\ell=1}^{-(t-t_0)}$} H_\zeta^{\ell-1} R_{\zeta,-1} \bveps_{\zeta, t+\ell}$ and a complementary series $\bfx_{c,\zeta,t} = H_\zeta^{\,t_0-t} \bfc_{\zeta,t_0}$ for all $t \in t_0 - {\mathbb N}$.  The Laurent series $[z I_X - H_\zeta]^{-1} = \mbox{$\sum_{\ell \in {\mathbb N}}$} z^{-\ell} {H_{\zeta}}^{\ell-1}$ is valid for $\lvert z \rvert > \lvert \zeta \rvert = 1$.  Therefore $\lim_{\ell \to \infty} \lVert H_\zeta^\ell \rVert^{1/\ell} = 1$ and so $\bfx_{c,\zeta,t}$ is subexponential for $t \in t_0 - {\mathbb N}$.  If $\bveps_{\zeta}$ is subexponential then $\bfx_{p,\zeta,t}$ and $\bfx_{\zeta,t}$ are also subexponential for $t \in t_0 - {\mathbb N}$.  Lemma~\ref{sar:lem2}~(iv) shows that $\bfx_{p,\zeta,t}$ can be rewritten as a sum of inverse powers of weighted truncated forward differences in the form
\begin{equation*}
\bfx_{p,\zeta,t} = \mbox{$\sum_{k \in {\mathbb N}}$} (-1)^k R_{\zeta,-k}\, {\bDelta_{\zeta}}^{-k} \bveps_{\zeta,t_0^-,t+k} \quad \mbox{ for all } t \in t_0 - {\mathbb N}.
\end{equation*}

\paragraph{Part $(iv)$.}  If $\infty \in \sigma$ the forward recursion $\bfx_{\infty,t} = R_{\infty,0}\bveps_{\infty,t} - G_\infty \bfx_{\infty,t-1}$ for all $t \in {\mathbb Z}$ is derived in Lemma~\ref{sar:lem1}~(ii).  Repeated application of the recursion gives
\begin{equation}
\label{mr:e8}
\bfx_{\infty, t} = \mbox{$\sum_{\ell = 0}^{t+r-1}$} R_{\infty,\ell} \bveps_{\infty, t-\ell} + (-1)^{r+t} G_\infty^{r+t} \bfx_{\infty, -r} \quad \mbox{ for all } r \in -t + {\mathbb N}
\end{equation}
where we have used the formula $R_{\infty,\ell} = (-1)^{\ell} G_\infty^{\ell} R_{\infty,0}$ for all $\ell \in {\mathbb N}-1$.  The singular part of the Laurent series for $R(z)$ at $\infty$ is given by $R_{\infty, \mbox{\scriptsize\textup{sg}}}(z) = \sum_{\ell \in {\mathbb N}-1} z^{\ell}R_{\infty,\ell}$ and is analytic for all $z \neq \infty$.  Therefore $\lim_{\ell \to \infty} \lVert R_{\infty,\ell} \rVert^{1/\ell} = 0$.   If $\bveps_{\infty}$ is subexponential it follows that the series $\sum_{\ell \in {\mathbb N}-1} R_{\infty,\ell} \bveps_{\infty, t-\ell}$ converges for all $t \in {\mathbb Z}$.  The limit as $r \to \infty$ in (\ref{mr:e8}) shows that the forward flow is given by
\begin{equation}
\label{mr:e9}
\bfx_{\infty,t} =  \mbox{$\sum_{\ell \in {\mathbb N}-1}$} R_{\infty, \ell}\, \bveps_{\infty,t-\ell} + (-1)^{t-t_0} G_\infty^{t-t_0} \bfc_{\infty, t_0, -\infty} \quad \mbox{ for all } t \in t_0 - 1 + {\mathbb N}
\end{equation}
where $\bfc_{\infty, t_0, -\infty} = \lim_{r \rightarrow \infty} (-1)^{r+t_0} G_\infty^{r+t_0} \bfx_{\infty,-r} \in X_\infty$.   The forward flow in (\ref{mr:e9}) is the sum of a particular cumulation series $\bfx_{p,\infty,t} = \mbox{$\sum_{\ell \in {\mathbb N}-1}$} R_{\infty, \ell}\, \bveps_{\infty,t-\ell}$ and a complementary series $\bfx_{c,\infty,t} = (-1)^{t-t_0} G_\infty^{t-t_0} \bfc_{\infty, t_0, -\infty}$.  The forward branch (\ref{mr:e9}) admits a two-sided extension if and only if $\bfc_{\infty, t_0, -\infty} \in {\mathcal C}_{G_\infty}(X_\infty)$.  The propagating operator $G_\infty$ is quasinilpotent and in particular cases may be nilpotent.  If $G_\infty$ is nilpotent then $\bfc_{\infty,t_0,-\infty} = \bfzero$ and $\bfx_{\infty,t}$ is uniquely determined for all $t \in {\mathbb Z}$.  If $G_\infty$ is quasinilpotent then $\bfx_{c,\infty,t}$ and hence also $\bfx_{\infty,t}$ may not be uniquely defined by (\ref{mr:e9}) for $t \in t_0 - {\mathbb N}$.  Thus the solution is not necessarily subexponential.  If $G_\infty$ is injective on ${\mathcal C}_{G_\infty}(X_\infty)$ then $\bfx_{c,\infty,t}$ and hence also $\bfx_{\infty,t}$ are uniquely determined for all $t \in {\mathbb Z}$.  In this case Lemma~\ref{mr:lem1}~(vi) shows that the flow is subexponential if and only if $\bfc_{\infty,t_0,-\infty} = \bfzero$.

\paragraph{Part $(v)$.}  If $0 \in \sigma$ the backward recursion $\bfx_{0,t} = R_{0,-1}\bveps_{0,t+1} - G_0 \bfx_{0,t+1}$ for all $t \in {\mathbb Z}$ is derived in Lemma~\ref{sar:lem1}~(iii).  Repeated application of the recursion gives
\begin{equation}
\label{mr:e10}
\bfx_{0, t} = \mbox{$\sum_{k=1}^{r-t}$} R_{0,-k} \bveps_{0, t+k} + (-1)^{r-t} G_0^{r-t} \bfx_{0, r} \quad \mbox{ for all } r \in t + {\mathbb N}
\end{equation}
where we have used the formula $R_{0,-k} = (-1)^{k-1} G_0^{k-1} R_{0,-1}$ for all $k \in {\mathbb N}$.  The singular part of the Laurent series for $R(z)$ at $0$ is given by $R_{0, \mbox{\scriptsize\textup{sg}}}(z) = \sum_{k \in {\mathbb N}} z^{-k}R_{0,-k}$ and is analytic for all $z \neq 0$.  Therefore $\lim_{k \to \infty} \lVert R_{0,-k} \rVert^{1/k} = 0$.   If $\bveps_0$ is subexponential it follows that the series $\mbox{$\sum_{k \in {\mathbb N}}$} R_{0,-k}\, \bveps_{0, t+k}$ converges for all $t \in {\mathbb Z}$.  The limit as $r \to \infty$ in (\ref{mr:e10}) shows that the backward flow is given by
\begin{equation}
\label{mr:e11}
\bfx_{0,t} =  \mbox{$\sum_{k \in {\mathbb N}}$} R_{0,-k}\, \bveps_{0,t+k} + (-1)^{t_0-t} G_0^{t_0-t} \bfc_{0, t_0, \infty} \quad \mbox{ for all } t \in t_0 + 1 - {\mathbb N}
\end{equation}
where $\bfc_{0, t_0, \infty} = \lim_{r \rightarrow \infty} (-1)^{r-t_0} G_0^{r-t_0} \bfx_{0,r} \in X_0$.   The backward flow is the sum of a particular cumulation series $\bfx_{p,0,t} = \mbox{$\sum_{k \in {\mathbb N}}$} R_{0,-k}\, \bveps_{0,t+k}$ and a complementary series $\bfx_{c,0,t} = (-1)^{t_0-t} G_0^{t_0-t} \bfc_{0, t_0, \infty}$.  The backward branch (\ref{mr:e11}) admits a two-sided extension if and only if $\bfc_{0, t_0, \infty} \in {\mathcal C}_{G_0}(X_0)$.  The propagating operator $G_0$ is quasinilpotent and in particular cases may be nilpotent.  If $G_0$ is nilpotent $\bfc_{0,t_0,\infty} = \bfzero$ and $\bfx_{0,t}$ is uniquely determined for all $t \in {\mathbb Z}$.  If $G_0$ is quasinilpotent $\bfx_{c,0,t}$ and hence also $\bfx_{0,t}$ may not be uniquely determined for $t \in t_0 + {\mathbb N}$.  Thus the solution is not necessarily subexponential.  If $G_0$ is injective on ${\mathcal C}_{G_0}(X_0)$ then $\bfx_{c,0,t}$ and hence also $\bfx_{0,t}$ are uniquely determined for all $t \in {\mathbb Z}$.   In this case Lemma~\ref{mr:lem1} shows that the flow is subexponential if and only if $\bfc_{0,t_0,\infty} = \bfzero$.   $\hfill \Box$ 

\begin{example}
\label{ex:1}
{\rm Define $X = L^2([0,1])$.  Let $A_0 = V \in {\mathcal B}(X)$ be the Volterra operator defined by $V \bfx(r) = \int_{[0,r]} \bfx(s) ds$ for all $ r \in [0,1]$ and let $A_1 = I \in {\mathcal B}(X)$ be the identity operator.  The operator $V$ is quasinilpotent and so $R(z) =  (V + zI)^{-1} = \sum_{k \in {\mathbb N}} (-1)^{k-1} z^{-k}V^{k-1}$ for all $z \neq 0$.  Therefore $R(z)$ is analytic everywhere except for an isolated essential singularity at $\zeta = 0$.  It follows that $\sigma = \{ 0 \}$.  If $V \bfx = \bfzero$ then $\bfx = \bfzero$.  Therefore $0 \in \sigma$ is not a generalized eigenvalue.  We have $P_0 = R_{0,-1}A_1 = I$ and so there is only one natural component $\bfx_0 = P_0 \bfx = \bfx$.  Let $t_0 = 0$.  The solution to (\ref{in:e1}) is a backward flow given by  
\begin{equation}
\label{mr:e12}
\bfx_{0,t} = \mbox{$\sum_{k \in {\mathbb N}}$} (-1)^{k-1} V^{k-1} \bveps_{0, t+k} + (-1)^t V^{-t} \bfc_{0, 0, \infty} \quad \mbox{ for all } t \in 1 - {\mathbb N}
\end{equation}
where $\bfc_{0, 0, \infty} = \lim_{r \to \infty}(-1)^rV^r \bfx_{0,r} \in {\mathcal C}_V(X)$.  If $V \bfx = \bfy \in X$ then $\bfy$ has an absolutely continuous representative with $\bfy(0) = \bfzero$ and $\bfx = \partial^1 \bfy$ where $\partial^{1} \bfy$ denotes the weak derivative.  Hence $V$ is injective.  There are infinitely many $\bfy \in L^2([0,1])$ with no weak derivative in $L^2([0,1])$ and so $V$ is not surjective.  The core is given by
\begin{equation*}
{\mathcal C}_V(X) = \mbox{$\bigcap_{n \in {\mathbb N}}$} V^n(L^2([0,1])) = \{ \bff \in \mbox{$\bigcap_{n \in {\mathbb N}}$} H^n([0,1]) \mid \partial^{\,r} \bff(0) = 0 \mbox{ for } r \in {\mathbb N} -1 \}
\end{equation*}	
and each $\bff \in {\mathcal C}_V(X)$ has a representative in ${\mathcal C}^{\infty}([0,1])$.  If $\bfy_0(0) = 0$ and $\bfy_0(s) = e^{-1/s}$ for $s \in (0,1]$ then $\partial^{\,r} \bfy_0(0) = 0$ for $r \in {\mathbb N}-1$ and so $\bfy_0 \in {\mathcal C}_V(X) \neq \{ \bfzero\}$.  Nevertheless $V^{-r} \bfc_{0,0,\infty} = \partial^{\,r} \bfc_{0,0,\infty} \in {\mathcal C}_V(X)$ is well defined as a weak derivative in the Sobolev sense for $r \in {\mathbb N}-1$.  Thus $\bfx_{0,t}$ is uniquely defined by (\ref{mr:e12}) for all $t \in {\mathbb Z}$.} $\hfill \Box$
\end{example} 

\begin{appendix}

\section{Supplementary algebraic results}
\label{a:sar}

\begin{lemma}
\label{sar:lem1}
Let $\zeta \in \sigma$.  The projected time series $\bfx_{\zeta,t}$ satisfies the following recursions.
\begin{enumerate}[(i)]
\item If $\zeta \neq 0, \infty$ then $\bfx_{\zeta,t} = -H_\zeta^{-1}R_{\zeta,-1}\bveps_{\zeta,t} + H_\zeta^{-1}\bfx_{\zeta,t-1}$ and $\bfx_{\zeta,t} = R_{\zeta,-1}\bveps_{\zeta,t+1} + H_\zeta \bfx_{\zeta,t+1}$ for all $t \in {\mathbb Z}$ where $H_\zeta = \zeta I_X - R_{\zeta,-1}A_{\zeta,0}$. 
\item If $\zeta = \infty$ then $\bfx_{\infty,t} = R_{\infty,0} \bveps_{\infty,t} - G_\infty \bfx_{\infty,t-1}$ for all $t \in {\mathbb Z}$ where $G_\infty = R_{\infty,0}A_1$.
\item If $\zeta = 0$ then $\bfx_{0,t} = R_{0,-1}\bveps_{0,t+1} - G_0 \bfx_{0,t+1}$ for all $t \in {\mathbb Z}$ where $G_0 = R_{0,-1}A_0$.
\end{enumerate} 
\end{lemma}

\textbf{Proof.}

(i) For $\zeta \neq 0, \infty$ the operator $R_{\zeta,-1}A_{\zeta,0}$ is quasinilpotent.  The Neumann expansion and the properties of the Laurent series coefficients  show that
\begin{align*}
{H_\zeta}^{-1}R_{\zeta,-1}A_0 & = [\zeta I_X - R_{\zeta,-1}A_{\zeta,0}]^{-1}R_{\zeta,-1}(A_{\zeta,0} - \zeta A_1) \\
& = [ \mbox{$\sum_{k \in {\mathbb N}}$}\, \zeta^{-k}(R_{\zeta,-1}A_{\zeta,0})^{k-1} ]R_{\zeta,-1}[A_{\zeta,0} - \zeta  A_1] \\
& = - R_{\zeta,-1}A_1 + \mbox{$\sum_{k \in {\mathbb N}}$}\, \zeta^{-k} [ (R_{\zeta,-1}A_{\zeta,0})^{k-1} R_{\zeta,-1}A_{\zeta,0} - (R_{\zeta,-1}A_{\zeta,0})^kR_{\zeta,-1}A_1 ] \\
& = - P_\zeta + \mbox{$\sum_{k \in {\mathbb N}}$}\, \zeta^{-k} (-1)^{k-1} [ R_{\zeta,-k}A_{\zeta,0} + R_{\zeta,-k-1} A_1 ] \\
& = -P_\zeta
\end{align*}
where we note from (\ref{mb:e2}) that $R_{\zeta,-k}A_{\zeta,0} + R_{\zeta,-k-1} A_1 = \bigzero$ for all $k \in {\mathbb N}$.  Now multiply (\ref{in:e1}) on the left by ${H_\zeta}^{-1} R_{\zeta,-1}$ to give $- P_\zeta \bfx_t + {H_\zeta}^{-1}R_{\zeta,-1}A_1 \bfx_{t-1} = {H_\zeta}^{-1} R_{\zeta,-1} \bveps_t$ for all $t \in {\mathbb Z}$.  It follows from $P_{\zeta} = R_{\zeta,-1}A_1$ and $R_{\zeta,-1} = R_{\zeta,-1}Q_\zeta$ that
\begin{equation}
\label{sar:e1}
- \bfx_{\zeta,t} + {H_\zeta}^{-1}\bfx_{\zeta,t-1} = {H_\zeta}^{-1}R_{\zeta,-1}\bveps_{\zeta,t} \quad \mbox{ for all } t \in {\mathbb Z}.
\end{equation}
Multiplying through by $H_\zeta$ and replacing $t$ by $t+1$ gives 
\begin{equation}
\label{sar:e2}
- H_\zeta \bfx_{\zeta,t+1} + \bfx_{\zeta,t} = R_{\zeta,-1}\bveps_{\zeta,t+1} \quad \mbox{ for all } t \in {\mathbb Z}.  \end{equation}
The desired recursions follow by rearranging (\ref{sar:e1}) and (\ref{sar:e2}).  

(ii) For $\zeta = \infty$ multiply (\ref{in:e1}) on the left by $R_{\infty,0}$ to give $R_{\infty,0}A_0 \bfx_t + R_{\infty,0}A_1 \bfx_{t-1} = R_{\infty,0}\bveps_t$ for all $t \in {\mathbb Z}$.  Now use $P_\infty = R_{\infty,0}A_0$, $R_{\infty,0}A_1 = R_{\infty,0}A_1[R_{\infty,0}A_0 + R_{\infty,-1}A_1] = R_{\infty,0}A_1R_{\infty,0}A_0$ because $R_{\infty,0}A_1R_{\infty,-1} = R_{\infty,0}Q_\infty^c = \bigzero$, and $R_{\infty,0} = R_{\infty,0}Q_\infty$ to deduce
\begin{equation}
\label{sar:e3}
\bfx_{\infty,t} + G_\infty \bfx_{\infty,t-1} = R_{\infty,0}\bveps_{\infty,t} \quad \mbox{ for all } t \in {\mathbb Z}.  \end{equation}
A simple rearrangement of (\ref{sar:e3}) gives the desired recursion. 

(iii) For $\zeta = 0$ multiply (\ref{in:e1}) on the left by $R_{0,-1}$ and replace $t$ by $t+1$ to give $R_{0,-1}A_0 \bfx_{t+1} + R_{0,-1}A_1\bfx_t = R_{0,-1}\bveps_{t+1}$ for all $t \in {\mathbb Z}$.  Now use $P_0 = R_{0,-1}A_1$ and $R_{0,-1} = R_{0,-1}Q_0$ to give
\begin{equation}
\label{sar:e4}
G_0 \bfx_{0,t+1} + \bfx_{0,t} = R_{0,-1}\bveps_{0,t+1} \quad \mbox{ for all } t \in {\mathbb Z}.
\end{equation}
A simple rearrangement of (\ref{sar:e4}) gives the desired recursion.  $\hfill \Box$

\begin{lemma}
\label{sar:lem2}
Suppose $\zeta \in \sigma$ with $\lvert \zeta\rvert \in (0,\infty)$.  The particular cumulation series $\bfx_{P, \zeta,t}$ can be written as a sum of inverse powers of weighted backward or forward differences.  The specific expressions are:
\begin{enumerate}[(i)]
\item if $\lvert \zeta \rvert \in (1,\infty)$ then $\bfx_{P, \zeta,t} = \mbox{$\sum_{k \in {\mathbb N}}$} (-1)^k R_{\zeta,-k}\, {\bnab_{\zeta}}^{-k} \bveps_{\zeta,t}$ for all $t \in {\mathbb Z}$;
\item if $\lvert \zeta \rvert \in (0,1)$ then $\bfx_{P, \zeta,t} = \mbox{$\sum_{k \in {\mathbb N}}$}\, (-1)^k R_{\zeta,-k}\, {\bDelta_{\zeta}}^{-k} \bveps_{\zeta, t+k}$ for all $t \in {\mathbb Z}$;
\item if $\lvert \zeta \rvert = 1$ then $\bfx_{p,\zeta,t} = \mbox{$\sum_{k \in {\mathbb N}}$} (-1)^k R_{\zeta,-k}\, {\bnab_{\zeta}}^{-k} \bveps_{\zeta,t_0^+,t}$ for all $t \in t_0 + {\mathbb N}$;
\item if $\lvert \zeta \rvert = 1$ then $\bfx_{p,\zeta,t} = \mbox{$\sum_{k \in {\mathbb N}}$} (-1)^k R_{\zeta,-k}\, {\bDelta_{\zeta}}^{-k} \bveps_{\zeta,t_0^-,t+k}$ for all $t \in t_0 - {\mathbb N}$.
\end{enumerate} 
\end{lemma}

\textbf{Proof.}  We refer readers to Section~\ref{s:mb} for details of the Laurent series expansions for $R(z)$.  Changes in the order of summation are justified because the series involved are absolutely convergent.  For (i) and (ii) with $\lvert \zeta \rvert \in (0,\infty)$ the Laurent series
\begin{equation*}
R_{\zeta, \mbox{\scriptsize\textup{sg}}}(z) = \mbox{$\sum_{k \in {\mathbb N}}$} (-1)^{k-1}(R_{\zeta,-1}A_{\zeta,0})^{k-1}R_{\zeta,-1}(z - \zeta)^{-k}
\end{equation*}
converges for all $ z \neq \zeta$.  Therefore $\lim_{k \to \infty} \lVert (R_{\zeta,-1}A_{\zeta,0})^k R_{\zeta,-1} \rVert^{1/k} \to 0$ as $k \to \infty$.  Note that $H_\zeta = \zeta I_X - R_{\zeta,-1}A_{\zeta,0}$.

\vspace{0.2cm}
(i) If $\lvert \zeta \rvert \in (1,\infty)$ then
\begin{align*}
\bfx_{p,\zeta,t} & = (-1) \mbox{$\sum_{\ell \in {\mathbb N}-1}$} ( \zeta I_X - R_{\zeta,-1}A_{\zeta,0})^{-\ell-1} R_{\zeta,-1} \bveps_{\zeta, t-\ell} \\
& = (-1) \mbox{$\sum_{\ell \in {\mathbb N}-1}$} \left\{ \mbox{$\sum_{k \in {\mathbb N}}$} \mbox{$\binom{\ell+k-1}{k-1}$} ( \zeta^{-1}R_{\zeta,-1}A_{\zeta,0})^{k-1} R_{\zeta,-1} \right\}\zeta^{-\ell-1} \bveps_{\zeta, t-\ell} \\
& = (-1) \mbox{$\sum_{k \in {\mathbb N}}$} (R_{\zeta,-1}A_{\zeta,0})^{k-1} R_{\zeta,-1} \left\{  \mbox{$\sum_{\ell \in {\mathbb N}-1}$} \mbox{$\binom{k-1+\ell}{\ell}$} \zeta^{-k-\ell} \bveps_{\zeta, t-\ell} \right\} \\
& = \mbox{$\sum_{k \in {\mathbb N}}$} (-1)^k R_{\zeta,-k}\, {\bnab_{\zeta}}^{-k} \bveps_{\zeta,t} \quad \mbox{ for all } t \in {\mathbb Z}.
\end{align*}

(ii) If $\lvert \zeta \rvert \in (0,1)$ then
\begin{align*}
\bfx_{p,\zeta,t} & = \mbox{$\sum_{\ell \in {\mathbb N}}$}\, ( \zeta I_X - R_{\zeta,-1}A_{\zeta,0})^{\ell-1} R_{\zeta,-1} \bveps_{\zeta, t+\ell} \\
& =  \mbox{$\sum_{\ell \in {\mathbb N}}$}\, \zeta^{\ell-1} \left\{ \mbox{$\sum_{k=1}^{\ell}$} \mbox{$\binom{\ell-1}{k-1}$} (-1)^{k-1}( \zeta^{-1}R_{\zeta,-1}A_{\zeta,0})^{k-1} R_{\zeta,-1} \right\} \bveps_{\zeta, t+\ell} \\
& = (-1) \mbox{$\sum_{k \in {\mathbb N}}$}\, (R_{\zeta,-1}A_{\zeta,0})^{k-1} R_{\zeta,-1} \left\{  (-1)^k \mbox{$\sum_{r \in {\mathbb N}-1}$} \mbox{$\binom{k-1+r}{r}$} \zeta^r \bveps_{\zeta, t+k+r} \right\} \\
& = \mbox{$\sum_{k \in {\mathbb N}}$}\, (-1)^k R_{\zeta,-k}\, {\bDelta_{\zeta}}^{-k} \bveps_{\zeta, t+k}  \quad \mbox{ for all } t \in {\mathbb Z}.
\end{align*}

(iii). If $\lvert \zeta \rvert = 1$ then
\begin{align*}
\bfx_{p,\zeta,t} & = (-1) \mbox{$\sum_{\ell=0}^{t-t_0-1}$} ( \zeta I_X - R_{\zeta,-1}A_{\zeta, 0})^{-\ell-1} R_{\zeta,-1} \bveps_{\zeta, t-\ell} \\
& = (-1) \mbox{$\sum_{\ell=0}^{t-t_0-1}$}\, \zeta^{-\ell-1} \mbox{$\sum_{k \in {\mathbb N}} \binom{\ell+k-1}{k-1}$} ( \zeta^{-1}R_{\zeta,-1}A_{\zeta,0})^{k-1} R_{\zeta,-1} \bveps_{\zeta, t-\ell} \\
& = (-1) \mbox{$\sum_{k \in {\mathbb N}}$} (R_{\zeta,-1}A_{\zeta,0})^{k-1} R_{\zeta,-1} \mbox{$\sum_{\ell=0}^{t-t_0-1}\,  \binom{k-1+\ell}{\ell}$} \zeta^{-k-\ell} \bveps_{\zeta, t-\ell} \\
& = \mbox{$\sum_{k \in {\mathbb N}}$} (-1)^k R_{\zeta,-k}\, {\bnab_{\zeta}}^{-k} \bveps_{\zeta,t_0^+,t} \quad \mbox{ for all } t \in t_0 + {\mathbb N}.
\end{align*}

(iv). If $\lvert \zeta \rvert = 1$ then
\begin{align*}
\bfx_{p,\zeta,t} & = \mbox{$\sum_{\ell=1}^{-(t-t_0)}$} ( \zeta I_X - R_{\zeta,-1}A_{\zeta,0})^{\ell-1} R_{\zeta,-1} \bveps_{\zeta, t+\ell} \\
& = \mbox{$\sum_{\ell=1}^{-(t-t_0)} \zeta^{\ell - 1} \sum_{k=1}^{\ell}\, \binom{\ell-1}{k-1}$}\, (-1)^{k-1} ( \zeta^{-1} R_{\zeta,-1}A_{\zeta,0})^{k-1}  R_{\zeta,-1} \bveps_{\zeta, t+\ell} \\
& = \mbox{$\sum_{k=1}^{-(t-t_0)}\, (-1)^{k-1} (R_{\zeta,-1}A_{\zeta,0})^{k-1}  R_{\zeta,-1} \sum_{\ell=k}^{-(t-t_0)}  \binom{\ell-1}{k-1}$}\, \zeta^{\ell - k} \bveps_{\zeta, t+\ell} \\
& = (-1) \mbox{$\sum_{k=1}^{-(t-t_0)}\, (R_{\zeta,-1}A_{\zeta,0})^{k-1}  R_{\zeta,-1} (-1)^k \sum_{r=0}^{-(t-t_0)-k}  \binom{k-1+r}{r}$}\, \zeta^r \bveps_{\zeta, t+k+r} \\
& = \mbox{$\sum_{k=1}^{-(t-t_0)}$}\, (-1)^k R_{\zeta,-k}\, {\bDelta_{\zeta}}^{-k} \bveps_{\zeta,t_0^-,t+k}  \quad \mbox{ for all } t \in t_0 - {\mathbb N}.
\end{align*}
In (iii) and (iv) the respective truncated forms of (\ref{mr:e0b}) and (\ref{mr:e0f}) remain valid when $\lvert \zeta \rvert = 1$.  $\hfill \Box$

\end{appendix}

\end{document}